\documentclass{amsart}
\usepackage{amssymb}

\usepackage{graphicx}
\usepackage{float}

\newtheorem{theorem}{Theorem}

\theoremstyle{remark}
\newtheorem{remark}[theorem]{Remark}

\numberwithin{equation}{section}

\newcommand\mL{L\kern-0.08cm\char39}

\begin{document}

\begin{large}

\title[]{A note on cycles in graphs with specified radius and diameter}

\author[P. Hrn\v ciar]{Pavel Hrn\v ciar}
\address{Department of Mathematics, Faculty of Natural Sciences,
          Matej Bel University, Tajovsk\'eho 40, 974 01 Bansk\'a Bystrica,
          Slovakia}
\email{Pavel.Hrnciar@umb.sk}




\begin{abstract}
Let $G$ be a graph of radius $r$ and diameter $d$ with $d\leq 2r-2$. We give a new proof that $G$ contains a cycle of  length at least $4r-2d$, i.e. for its circumference it holds $c(G)\geq 4r-2d$.
\end{abstract}

\maketitle



For a connected graph $G$, the \emph{distance} $d(u,v)$ between vertices $u$ and $v$ is the length of a shortest path joining them. The distance between a vertex $u\in V(G)$ and a subgraph $H$ of $G$ will be denoted by $d(u,H)=\min\{d(u,v); v\in V(H)\}$.
The \emph{eccentricity} $e(u)$ of a vertex $u$ of $G$ is the distance of $u$ to a vertex farthest from $u$ in $G$, i.e. $e(u)=\max\{d(u,v); v\in V(G)\}$.
The \emph{radius} $r(G)$ and the \emph{diameter} $d(G)$ of $G$ are the minimum and the maximum eccentricity of its vertices, respectively.
The \emph{circumference} of a graph $G$, denoted $c(G)$, is the length of any longest cycle in $G$.

A nontrivial connected graph with no cut-vertices is called a \emph{nonseparable graph}. A \emph{block} of a graph $G$ is a maximal nonseparable subgraph of $G$. If $u$ is a cut-vertex of $G$ belonging to a block $B$ of $G$ then $V_{u,B}$ (or briefly $V_u$) will denote the set $\{v\in V(G); d(v,u)=d(v,B)\}$.

The aim of this note is to give a new proof of the following theorem (see~\cite{Hr}). This proof is very different from the previous one and it is also simpler.

\begin{theorem}
Let $G$ be a graph of radius $r$ and diameter $d$ with $d\leq 2r-2$. Then $c(G)\geq 4r-2d$.
\end{theorem}

\begin{proof}
Since $d\leq 2r-2$, $G$ is not a tree. Suppose, contrary to our claim, that $c(G)< 4r-2d$. If $B$ is a block of $G$, $B\neq K_2$, then every two vertices of $B$ lie on a common cycle (see~\cite{BH}, Theorem 1.6) of length less than $4r-2d$. Hence we get $d(B)< 2r-d$. We distinguish three cases.

\begin{itemize}
\item[(1)] 
There exists a block $B$ of $G$ with $d(v,B)\leq d-r$ for every vertex $v\in V(G)$. 
\newline
Let $w\in V(B)$. If $B=K_2$ then $e(w)\leq 1+(d-r)\leq 1+(r-2)<r$, a contradiction. If $B\neq K_2$ then $e(w)\leq d(B)+(d-r)<(2r-d)+(d-r)=r$, which is again a contradiction.

\item[(2)]
There exists a block $B$ of $G$ containing two cut-vertices $u_1$, $u_2$ with $\max\{d(v,u_i); v\in V_{u_i,B}\}> d-r$ for $i=1,2$.
\newline 
Let $v\in V(G)$ be a vertex with $d(v,u_1)=r$. We distinguish two subcases.
\newline
(i) $v\in V_{u_1}.$
\newline
Let $w\in V_{u_2}$ be a vertex with $d(w,u_2)> d-r$. We get $d(v,w)=d(v,u_1)+d(u_1,u_2)+d(u_2,w)>r+1+(d-r)>d$, a contradiction.
\newline
(ii) $v\notin V_{u_1}.$
\newline
Consider a vertex $w\in V_{u_1}$ with $d(w,u_1)>d-r$. We get $d(w,v)=d(w,u_1)+d(u_1,v)> (d-r)+r=d$, a contradiction too.
\item[(3)] 
In every block $B$ of $G$ there exists exactly one cut-vertex $u$ with $\max\{d(v,B); v\in V_{u,B}\}> d-r$. 
\\
Choose vertices $x$ and $y$ such that $d(x,y)=d(G)$. Let $P$ by a shortest path joining $x$ and $y$. Let $w$ be a vertex of $P$ with $|d(x,w)-d(y,w)|\leq 1$, i.e. $d(x,w)\leq r-1$ and $d(y,w)\leq r-1$. Consider  a block $B$ of $G$ such that $w\in V(B)$ and $|V(B)\cap V(P)|\geq 2$. Let $u\in V(B)$ be a vertex with $d(x,u)=d(x,B)$ and $v\in V(B)$ be a vertex with $d(y,v)=d(y,B)$. Obviously, $u\neq v$ and $u$, $v$ are vertices of $P$. Without loss of generality it suffices to distinguish two subcases.
\newline(i) $d(x,B)>d-r.$
\newline The vertex $u$ is a cut-vertex of $G$ and it holds $d(z,B)\leq d-r$ for every vertex $z\in V(G)- V_u$. If for every vertex $z\in V_u$ it holds $d(z,w)\leq r-1$ then it is easy to check (see the case (1)) that $e(w)<r$, a contradiction. If there exists a vertex $z\in V_u$ with $d(z,w)\geq r$ we have
$$
d(z,y)=d(z,w)+d(w,y)\geq r+d(w,y)> d(x,w)+d(w,y)= d(x,y)=d,
$$
again a contradiction.
\newline
(ii) $d(x,B)\leq d-r$ and $d(y,B)\leq d-r.$
\newline If $B=K_2$ then we have
$$ 
d(x,y)\leq 2(d-r)+1=d+1+(d-2r)\leq d+1+(-2)< d,
$$
a contradiction.
\newline If $B\neq K_2$ then we get 
$$
d(x,y)=d(x,B)+d(B)+d(y,B)< 2(d-r)+(2r-d)=d,
$$
a contradiction.
\end{itemize}
\end{proof}

\begin{remark}
 For $r\geq 3$ the bound $4r-2d$ in Theorem 1 is the best possible (see ~\cite{Hr}).
\end{remark}

\begin{remark}
 For a graph $G$ with radius $r$, diameter $d\leq 2r-2$, with at most $3r-2$ vertices, it holds $c(G)\geq 2r$ (see~\cite{Ha}).
\end{remark}

\end{large}


\begin{thebibliography}{99}


\bibitem{BH}
   F. Buckley and F. Harary,
   \emph{Distance in Graphs}, 
   Addison-Wesley Publishing Company, Redwood City, CA, 1990.
   
\bibitem{Ha}
   A. Haviar, P. Hrn\v ciar and G. Monoszov\'a,  
   \emph{Eccentric sequences and cycles in graphs},
   Acta Univ. M. Belii, Ser. Math.~no 11~(2004), 7--25.

\bibitem{Hr}
   P. Hrn\v ciar, 
   \emph {On cycles in graphs with specified radius and diameter},
   Acta Univ. M. Belii, Ser. Math.~no 20~(2012),
   7--10.

\end{thebibliography}
\end{document}